\documentclass[10pt,a4,leqno]{amsart}
\usepackage{amsthm}
\usepackage{amsmath}
\usepackage{amssymb}
\usepackage{graphicx}
\usepackage{bbm}
\usepackage[%dvipdfmx,
colorlinks=true, citecolor=blue, linkcolor=blue, urlcolor=red]{hyperref}
\usepackage{color}
\usepackage[all]{xy}
\allowdisplaybreaks

\makeatletter

\@addtoreset{equation}{section}
\makeatother

\makeatletter
\@namedef{subjclassname@2020}{%
  \textup{2020} Mathematics Subject Classification}
\makeatother

\newtheoremstyle{my theoremstyle}
{1.0em}                    % Space above
    {1.0em}                    % Space below
    {\itshape}                   % Body font
    {}                           % Indent amount
    {\scshape}                   % Theorem head font
    {.}                          % Punctuation after theorem head
    {.5em}                       % Space after theorem head
    {}  % Theorem head spec (can be left empty, meaning ‘normal’)

\newtheoremstyle{dfn}
{1.0em}                    % Space above
    {1.0em}                    % Space below
    {}                   % Body font
    {}                           % Indent amount
    {\scshape}                   % Theorem head font
    {.}                          % Punctuation after theorem head
    {.5em}                       % Space after theorem head
    {}  % Theorem head spec (can be left empty, meaning ‘normal’)
    
\theoremstyle{my theoremstyle}
   \newtheorem{thm}{Theorem}[section]
   \newtheorem{lem}[thm]{Lemma}
   \newtheorem{prop}[thm]{Proposition}
   \newtheorem{cor}[thm]{Corollary}
   
\theoremstyle{dfn}

\theoremstyle{remark}   
   
   \newtheorem{rmk}[thm]{{\scshape Remark}}

\newcommand{\C}{\mathbb{C}}
\newcommand{\Q}{\mathbb{Q}}
\newcommand{\Z}{\mathbb{Z}}

\newcommand{\R}{\mathbb{R}}

\newcommand{\Ch}{\operatorname{CH}}
\newcommand{\jac}{\operatorname{Jac}}

\numberwithin{equation}{section}

\date{\today}

\begin{document}
\title{Non-torsion algebraic cycles on the Jacobians of Fermat quotients}
\author{Yusuke Nemoto}
\date{\today}
\address{Graduate School of Science and Engineering, Chiba University, 
Yayoicho 1-33, Inage, Chiba, 263-8522 Japan.}
\email{y-nemoto@waseda.jp}

\maketitle

\begin{abstract}
We study the Abel-Jacobi image of the Ceresa cycle $W_{k, e}-W_{k, e}^-$, where $W_{k, e}$ is the image of the $k$th symmetric product of a curve $X$ with a base point $e$ on its Jacobian variety. 
For certain Fermat quotient curves of genus $g$, we prove that for any choice of the base point and $k \leq g-2$, the Abel-Jacobi image of the Ceresa cycle is non-torsion.  
In particular, these cycles are non-torsion modulo rational equivalence. 
\end{abstract}

\section{Introduction}
Let $X$ be a smooth projective curve of genus $g$ over $\C$ and $\jac(X)$ be its Jacobian. 
Let $\Ch_k(\jac(X))_{\hom}$ be the Chow group of   
homologically trivial algebraic cycles of dimension $k$ on $\jac(X)$ modulo rational equivalence. 
To study this group, we consider the Abel-Jacobi map
 $$\Phi_k \colon \Ch_k(\jac(X))_{\rm hom} \to J_k(\jac(X))\quad (k=1, \ldots, g-1).  $$
Here, $J_k(\jac(X))$ is a complex torus, which is called the Griffiths intermediate Jacobian (see Section 3.1). 
It is well known that $\Phi_{g-1}$ is an isomorphism by the Abel-Jacobi theorem, however, for a general $k$, $\Phi_k$ is neither injective nor surjective.  
Fix a base point $e \in X$ and let $\iota_e$ be the emmbedding defined by 
$$\iota_e \colon X \to \jac(X); \quad x \mapsto [x]-[e]. $$
Put $X_e=\iota_e(X)$. 
We denote $X_e^-$ by the image of $X_e$ under the inversion map.  
%Since the action of $[-1]$ on the cohomology of even degree is trivial, we have
Since the inversion map acts trivially on the cohomology groups of even degree, we have 
 $$X_e-X_e^- \in \Ch_1(\jac(X))_{\rm hom}. $$
 Let $W_{k, e}$ be the image of the $k$th symmetric product of $X$ on $\jac(X)$. 
 As in the case of $k=1$, we have 
 $$W_{k,e} - W_{k,e}^- \in \Ch_k(\jac(X))_{\rm hom}. $$
 These cycles are called the Ceresa cycles and for a generic curve $X$, Ceresa \cite{Ceresa} proves that if $1 \leq k \leq g-2$, then $W_{k,e} - W_{k,e}^-$ is non-trivial modulo algebraic equivalence. 
 
 For a positive integer $N$ and integers $a, b \in \{1, \ldots, N-1\}$, 
 let $C_N^{a, b}$ be the smooth projective curve birational to the affine curve
 $$y^N=x^a(1-x)^b. $$
 Let $F_N$ be the Fermat curve of degree $N$. 
Then $C_N^{a, b}$ is a quotient of $F_N$ by a cyclic group $G_N^{a, b}$ (see Section 2).  
Let $g$ be the genus of $C_N^{a, b}$. 
The main theorem of this paper is as follows. 

\begin{thm} \label{main}
Suppose that $N$ has a prime divisor $p>7$ such that 
$p \nmid ab$ and $a^2+ab+b^2 \equiv 0 \pmod{p}$.  
Then $\Phi_k(W_{k, e}-W_{k, e}^-) \in J_k(\jac(C_N^{a, b}))$ is non-torsion for any choice of the base point $e \in C_N^{a, b}$ and $k=1, \ldots, g-2$. 
\end{thm}

 \begin{rmk}
 When $N$ does not have a prime divisor $p>7$, there exist some examples that the Abel-Jacobi image of the Ceresa cycle of $C_N^{a, b}$ is torsion.  
For example, $\Phi_1(X_e-X_e^-)$ is torsion for $X=C_{9}^{1, 2}$, $C_{12}^{1, 3}$, $C_{15}^{1, 5}$ and $e=(0, 0)$ (\cite[\S2 Theorem]{B}, \cite[Theorem 3.2]{LS}).
 \end{rmk}

The algebraical non-triviality of the Ceresa cycles of $F_N$ ($N \leq 1000$) and $C_p^{1, b}$ ($p \leq 1000$ is a prime and $b^2+b+1 \equiv 0 \pmod{p}$) 
is proved by Harris \cite{Harris1}, Bloch \cite{Bloch}, Kimura \cite{Kimura}, Tadokoro \cite{Tadokoro1, Tadokoro2, Tadokoro3}, and Otsubo  \cite{Otsubo}. 
 Moreover, Otsubo \cite{Otsubo} and Tadokoro \cite{Tadokoro3} give a sufficient condition for the Ceresa cycles of these to be non-torsion modulo algebraic equivalence,  
  however, it is impossible to confirm numerically these conditions. 
 There are only two explicit examples of non-torsioness modulo algebraic equivalence for $k=1$: $F_4$ by Bloch \cite{Bloch} and $C_7^{1, 2}$ by Kimura \cite{Kimura}.

% Let $p>7$ be a prime and $N$ be a positive integer which has the prime divisor $p$. 
Let $N$ be a positive integer divisible by a prime  $p > 7$. 
 Eskandari-Murty \cite{EM2, EM} prove that $\Phi_1(F_{N, e} -F_{N, e}^-)$ is non-torsion for any $e \in F_N$, 
in particular, $F_{N, e} -F_{N, e}^-$ is non-torsion modulo rational equivalence. 
 Moreover, they conjecture that the same result holds for $C_p^{1, m}$ with $m \in \{1, \ldots, p-2\}$ and $m \neq 1, (p-1)/2, p-2$ \cite[Section 4, Remark (2)]{EM}. 
Theorem \ref{main} partially but affirmatively answers their conjecture.

We briefly give a sketch of the proof. 
First, we reduce to the case $k=1$ using a method of Otsubo \cite{Otsubo}  (see Proposition \ref{red}). 
The reduction to the case $N=p$ is easy. 
%Specifically, via the isomorphism 
%$$J_k(\jac(C_N^{a, b})) \cong {\rm Hom}(\wedge^{2k+1} H^1(C_N^{a, b}, \Z), \R/\Z), $$
%we regard $\Phi_k(W_{k, e}-W_{k, e}^-)$ as a function on $\wedge^{2k+1} H^1(C_N^{a, b}, \Z)$.  
%By evaluating the function at a special element $\varphi_1 \wedge \cdots \wedge \varphi_{2k+1}$, 
%%By taking up a special element $\varphi_1 \wedge \cdots \wedge \varphi_{2k+1}$, 
%the non-torsionness for a general $k$ is reduced to that for $k=1$ (see Proposition \ref{red}). 
The rest of the proof is parallel to the method of Eskandari-Murty \cite{EM2, EM}. 
First, the Abel-Jacobi image of the Ceresa cycle is described by an extension of mixed Hodge structures by Harris \cite{Harris1} and Pulte \cite{Pulte} (see Section 3.2).  
Secondly, we construct a $1$-cycle $Z$ on $C_p^{a,b} \times C_p^{a,b}$ and evaluate the extension of mixed Hodge structures at $Z$. 
Here, we use the assumptions on $a$ and $b$ so that an automorphism of $F_p$ of order $3$ descends to $C_p^{a, b}$.  
Then it is expressed by a rational point $P_Z \in \jac(C_N^{a, b})$ 
by formulas of Kaenders \cite{Ka} and Darmon-Rotger-Sols \cite{DRS} (see Sections 3.3, 3.4). 
Finally, since $P_Z$ is non-torsion by a result of Gross-Rohrlich \cite{GR} (see Section 2), where we use the assumption $p > 7$, the theorem follows.

\section{Fermat quotient curves}
Let $N >3$ be an integer and for integers $a, b\in \{1, \ldots, N-1\}$,  %satisfying $0 < a, b, a+b < N$ and $\gcd(N, a, b, a+b)=1$, 
let $C^{a, b}_N$ be the smooth projective curve birational to
$$y^N=x^a(1-x)^b. $$
The map 
$$C^{a, b}_N \to \mathbb{P}^1; \quad (x, y) \mapsto x$$
is ramified at $x=0, 1$ and $\infty$. 
Above $0$ (resp. $1$, $\infty$), there are $\gcd(N, a)$ (resp. $\gcd(N, b)$, $\gcd(N, a+b)$) branches and the ramification index is $N/\gcd(N, a)$ (resp. $N/\gcd(N, b)$, $N/\gcd(N, a+b)$). 
Therefore by the Riemann-Hurwitz formula, the genus of $C^{a, b}_N$ is 
$$\dfrac12(N-(\gcd(N, a)+\gcd(N, b)+\gcd(N, a+b)))+1. $$ 
We have an isomorphism 
\begin{align*} %\label{isom}
C^{a, b}_N \cong C^{b, a}_N
\end{align*}
sending $x$ to $1-x$. 
If two other integers $a', b' \in \{1, \ldots, N-1\}$ satisfy the relation
$$(a', b') = (ha, hb) +(Ni, Nj)$$
for some integers $h, i, j$ with $\gcd(N, h)=1$,  
we have 
\begin{align*} %\label{isom} 
C^{a, b}_N \cong C^{a', b'}_N; \quad (x, y) \mapsto (x, y^h x^i(1-x)^j). 
\end{align*}
Let $F_N$ be the Fermat curve of degree $N$ defined by 
$$u^N+v^N=w^N. $$
%Its genus is $(N-1)(N-2)/2 \geq 3$ by our assumption 
%and its affine equation is written as 
%$$u^N+v^N=1 \quad (u=u/w, \ v=v/w). $$
Then there is a morphism 
 $$\pi_N^{a, b} \colon F_N \to C^{a, b}_N; \quad (u : v : w) \mapsto (x, y)=(u^Nw^{-N}, u^av^bw^{-a-b}).$$
Define a finite group by
$$G_N=\Z/N\Z \oplus \Z/N\Z$$
and denote an element $(r, s) \in G_N$ by $g_N^{r, s}$. 
Fix a primitive $N$-th root of unity $\zeta_N$ and     %\in \mu_N$. 
let $G_N$ act on $F_N$ by
$$g_N^{r, s}(u : v : w)=(\zeta_N^r u : \zeta_N^s v : w). $$
  Let $G_N^{a, b}$ be a subgroup of $G_N$ defined by 
 $$G_N^{a, b}= \{g_N^{r, s} \in G_N \mid ar+bs=0\}.  $$
If $\gcd(N, a, b)=1$, $F_N$ is 
generically Galois over $C^{a, b}_N$ and 
$$\operatorname{Gal} (F_N/ C^{a, b}_N)=G_N^{a, b} =\langle g_N^{b, -a} \rangle \simeq \Z/N\Z. $$   
There is an automorphism $\alpha$ of $F_N$ of order $2$ defined by
\begin{align*} %\label{a}
\alpha((u : v : w)) = (v : u : w). 
\end{align*}
When $N$ is odd, there is an automorphism $\beta$ of $F_N$ of order $3$ defined by % be an automorphism of $F_N$ defined by
\begin{align*} %\label{a}
\beta((u : v : w)) = (-v : w : u). 
\end{align*}

\begin{lem}[cf. {\cite[Section 3.1]{IS}}] \label{alpha} 
Suppose that $\gcd(N, a, b)=1$. Then: 
\begin{enumerate}
\item
$\alpha$ descends to $C^{a, b}_N$ if and only if $a^2 \equiv b^2  \pmod{N}$.
\item Suppose that $N$ is odd.   
Then $\beta$ descends to $C^{a, b}_N$ if and only if $a^2+ab+b^2 \equiv 0 \pmod{N}$. 
We denote $\widetilde{\beta}$ by this automorphism.  
\end{enumerate}
\end{lem}

\begin{proof}
We only prove (ii) since we use the morphism $\widetilde{\beta}$ to prove Theorem \ref{main} and (i) is similarly proved. 
The automorphism $\beta$ descends to $\widetilde{\beta}$ if and only if 
$$\pi_N^{a, b}( \beta (g_N^{b, -a}  (u : v : w)))=\pi_N^{a, b} (\beta(u : v : w)), $$
i.e.  there exists an integer $i$ such that 
$$\left(-\zeta_N^{-a} v :  w : \zeta_N^{b} u \right)=\left(-\zeta_N^{bi}v :  \zeta_N^{-ai}w : u\right). $$
This is equivalent to 
\begin{align} \label{equiv}
a+b \equiv -bi \quad {\rm and} \quad b \equiv ai \pmod{N}. 
\end{align}
First, \eqref{equiv} implies $a^2+ab+b^2 \equiv 0 \pmod{N}$. 
On the other hand, if $a^2+ab+b^2 \equiv 0 \pmod{N}$, then we have $\gcd(N,a)=\gcd(N, b)=1$ by the assumption $\gcd(N, a, b)=1$. 
Therefore, there is an integer $i$ such that $ai \equiv b \pmod{N}$, which satisfies \eqref{equiv}. 
\end{proof} 
\begin{rmk} \label{rmk1} \ 

\begin{enumerate}
\item If $N$ is a prime, the condition $a^2+ab+b^2 \equiv 0 \pmod{N}$ implies that $N \equiv 1 \pmod{3}$. 
%\item When $N=7$, the curve $C^{a, b}_7$ with $a^2+ab+b^2 \equiv 0\pmod{7}$ is isomorphic to the Klein quartic 
%$$u^3v+v^3w+w^3u=0$$
%whose genus is $3$. 
\item When $N=a^2+ab+b^2$, the curve $C^{a, b}_N$ is isomorphic to the Hurwitz curve (\cite[Lemma 3.8]{IS}) which is the smooth projective curve birational to 
 $$X^{b}Y^{a+b}+Y^{b}Z^{a+b}+Z^{b}X^{a+b}=0.$$
\end{enumerate}
\end{rmk}
When $\gcd(N, 6)=1$, 
the automorphism $\beta$ of $F_N$ has two fixed points
$$S=(\zeta_6 : \zeta_6^{-1} :1), \quad \overline{S}=(\zeta_6^{-1} : \zeta_6 : 1). $$
%are points on $F_N$, which are only fixed points under the automorphism $\alpha$. 
\begin{lem} \label{fix}
Suppose that $\gcd(N, a, b)=\gcd(N, 6)=1$ and $a^2+ab+b^2 \equiv 0 \pmod{N}$.   
Then the fixed points of the automorphism $\widetilde{\beta}$ of $C^{a, b}_N$ are $\pi_N^{a, b}(S)$ and $\pi_N^{a, b}(\overline{S})$, which are distinct. 
\end{lem}

\begin{proof}
%By our assumption, we have $N \equiv 1 \pmod{6}$, hence $P$, $\overline{P} \in F(N)$. 
Let $\pi_N^{a, b}(Q) \in C^{a, b}_N$ ($Q \in F_N$) be a fixed point of $\widetilde{\beta}$, i.e.   
there exists an integer $i$ such that 
\begin{align*} %\label{fix2}
\beta(Q)= g_N^{bi, -ai}  Q. 
\end{align*}
Since the order of $\beta$ is $3$, we have %\eqref{fix2} is equivalent to 
$$\beta^N(Q)=\beta^{\pm 1}(Q)=Q, $$
which implies that $Q=S$ or $\overline{S}$. 
We are to show that $\pi_N^{a, b}(S) \neq \pi_N^{a, b}(\overline{S})$. 
Suppose that $\pi_N^{a, b}(S) = \pi_N^{a, b}(\overline{S})$, i.e.  
there exists an integer $i$ such that 
$$\zeta_6 = \zeta_N^{bi}\zeta_6^{-1}, \quad \zeta_6^{-1}= \zeta_N^{-ai} \zeta_6. $$
Therefore we have $\zeta_6^{2N}=1$, which contradicts the assumption $\gcd(N, 6)=1$. 
\end{proof}

Put $P_0=(0 : 1: 1) \in F_N$ and let $F_N \to \jac(F_N)$  be a map defined by 
$Q \mapsto[Q]-[P_0]$. 
Similarly, we define a map $C^{a, b}_N \to \jac(C^{a, b}_N)$ by sending $Q'$ to $[Q'] - [\pi_N^{a, b}(P_0)]$. 
Then we have a commutative diagram 
\[
   \xymatrix{
    F_N \ar[r] \ar[d]_{\pi_N^{a, b}} & \jac(F_N) \ar[d]^{(\pi_N^{a, b})_*} \\
   C^{a, b}_N   \ar[r]  & \jac(C^{a, b}_N). 
} \]
%Let $e$ be a cusp of $F_N$ (i.e.  a point satisfying $x_0y_0z_0=0$). 
The following result of Gross and Rohrlich is one of the key ingredients to the proof of Theorem \ref{main}. 
\begin{thm}[{\cite[Theorem 2.1]{GR}}] \label{GR}
Let $N$ be an integer such that $\gcd(N, 6)=1$ and divisible by a prime $p>7$. 
If $a-b, a+2b, 2a+b  \not \equiv 0 \pmod{p}$, then the point $(\pi_N^{a, b})_*([S]+[\overline{S}] -2[P_0])$ on $\jac(C^{a, b}_N)$ is non-torsion. 
\end{thm}

\section{Algebraic cycles and Hodge theory of quadratic iterated integrals}
\subsection{Extension of mixed Hodge structures}
Let $R=\Z$ or $\Q$. 
An $R$-mixed Hodge structure $H$ is an $R$-module $H_R$ of finite rank equipped with an increasing weight filtration $W_{\bullet}$ on $H_{\Q}:=H_R \otimes_{R} \Q$
and a decreasing Hodge filtration $F^{\bullet}$ on $H_{\C}:=H_R \otimes_{R} \C$ such that for each $k$, 
${\rm Gr}_k^W(H_{\Q})$ with the induced filtration $F^{\bullet}$ is a pure $\Q$-Hodge structure of weight $k$. 
%Given an $R$-mixed Hodge structure $H$, we put $H_{\Q}:=H_{R} \otimes_{R} \Q$, $H_{\C}:=H_{R} \otimes_{R} \C$, respectively. 
%Let $R(n)$ be the pure $R$-Hodge structure of weight $-2n$ with $H_{R}=R$, $H_{\C}=H^{-n, -n}$ and the shifted filtration 
%$F^pR(n)_{\C}=F^{p+n}R_{\C}$ and $W_kR(n)_{\Q}=W_{k+2n}R_{\Q}$.  
%be the Yoneda extension groups in MHS$(R)$. 
%For an $R$-mixed Hodge structure $H$,
Let $R(n)$ be the Tate object of pure weight $-2n$ and  put $H(n)=H\otimes_{R} R(n)$. 
Let $H^{\vee}$ be the dual $R$-mixed Hodge structure of $H$.  

Let MHS$(R)$ be the category of $R$-mixed Hodge structures. 
For $R$-mixed Hodge structures $A$, $B$, 
let ${\rm Ext}_{{\rm MHS}(R)}(A, B)$ denote the set of equivalence classes of extensions of $R$-mixed Hodge structures, i.e. exact sequences 
$$0 \to B \to E \to A \to 0$$
of $R$-mixed Hodge structures up to natural equivalence relation.  %and the Baer sum.  
There is a natural operation called the Baer sum which makes ${\rm Ext}_{{\rm MHS}(R)}(A, B)$ an abelian group. 
If $X$ is a smooth projective variety over $\C$, the cohomology group $H^n(X, \Z)$ underlies a pure $\Z$-Hodge structure of weight $n$, which we denote by $H^n(X)$.  

For a pure $\Z$-Hodge structure $H$ of weight $-1$, 
the intermediate Jacobian is defined by 
$$JH=H_{\C}/(F^0H_{\C} +H_{\Z}), $$
which is a complex torus. 
We have Carlson's isomorphism \cite{Carlson} 
$$JH \cong {\rm Ext}_{{\rm MHS}(\Z)} (H^{\vee}, \Z).$$ 
For a smooth projective variety $X$ over $\C$, $H_{2k+1}(X)(-k)$ is a pure $\Z$-Hodge structure of weight $-1$, and
$$J_k(X):=JH_{2k+1}(X)(-k) \cong (F^{k+1}H^{2k+1}(X, \C))^{\vee} /H_{2k+1}(X, \Z) $$
is the $k$-th intermediate Jacobian of Griffiths.  
The Carlson isomorphism is written as 
$$J_k(X) \cong {\rm Ext}_{{\rm MHS}(\Z)}(H^{2k+1}(X)(k), \Z). $$

Let $\Ch_k(X)$ be the Chow group of $k$-dimensional algebraic cycles on $X$ modulo rational equivalence, and $\Ch_k(X)_{\rm hom}$ be the subgroup of homologically trivial cycles. 
Then 
we have the Abel-Jacobi map 
$$\Phi_k \colon \Ch_k(X)_{\rm hom} \to J_k(X); \quad Z  \mapsto \left(\eta \mapsto \int_{\Gamma} \eta \right)$$
for any $\eta \in F^{k+1}H^{2k+1}(X, \C)$, where $\Gamma$ is a topological $(2k+1)$-chain such that $\partial \Gamma = Z$. 

From now on, let $X$ be a smooth projective curve of genus $g \geq 3$ over $\C$. 
Let  
$$\langle \ , \ \rangle \colon H^1(X) \otimes H^1(X) \to H^2(X)= \Z(-1)$$
be the cup product $\varphi \otimes \varphi' \mapsto \int_X \varphi \wedge \varphi'$. 
Choosing a base point $e \in X$, $X$ is embedded into $\jac(X)$ sending $e$ to zero. 
It induces isomorphisms
$$H_1(X) \xrightarrow{\simeq} H_1(\jac(X)), \quad H^1(\jac(X)) \xrightarrow{\simeq} H^1(X),$$
which do not depend on the choice of $e$. We identify these and denote them by $H_1$ and $H^1$, respectively.  
Recall that the cup product induces an isomorphism
$$\wedge^nH^1 \xrightarrow{\simeq} H^n(\jac(X)). $$

For $e \in X$, let $\iota_e \colon X \to \jac(X)$ be a map defined by $P \mapsto [P]-[e]$. 
Let $X^k$ (resp. $\jac(X)^k$) be the $k$-fold product of $X$ (resp. $\jac(X)$) and 
$\mu \colon \jac(X)^k \to \jac(X)$ be the addition. 
%Then we have a commutative diagram 
%
%$$
%\xymatrix{
%X^k\ar[r]^-{\mu \circ (\iota_e)^k}\ar[d]_-{\pi}&\jac(X)\\
%X^k/S_k \ar[ru]_-{\overline{\iota_e}}
%}
%$$
%where $S_k$ is the symmetric group of degree $k$ and $\pi$ is the natural surjection. 
We put 
$$W_{k, e}=(\mu \circ (\iota_e)^k)(X^k) \quad (1 \leq k \leq g). $$
Then $W_{k, e}$ defines an algebraic $k$-cycle on $\jac(X)$ and $W_{k, e}-W_{k, e}^-$ defines an element of $\Ch_k(\jac(X))_{\rm hom}$.

%Let $X^k$ (resp. $\jac(X)^k$) be the $k$-fold product of $X$ (resp. $\jac(X)$) and $S_k$ be the symmetric group of degree $k$. 
%Then we have a commutative diagram 
%\[
%   \xymatrix{
%    X^k \ar[r]^{(\iota_e)^k} \ar[d]_{\pi} & \jac(X)^k \ar[d]^{\mu} \\
%   X^k/S_k   \ar[r]^{\overline{(\iota_e)^k}}  & \jac(X),  
%} \]
%where $\pi$ is the natural surjection, $\mu$ is the addition and $\overline{(\iota_e)^k}$ is the induced map. 
%Then we define 

 \begin{prop} \label{red} 
%Let $X$ be a smooth projective curve of genus $g \geq 3$ over $\C$. 
 If $\Phi_1(X_e -X_e^-)$ is non-torsion, then $\Phi_k(W_{k, e}-W_{k, e}^-)$ is non-torsion for any $k=2, \ldots, g-2$.  
 \end{prop}

 \begin{proof}
 Let $S=\{e_i, f_i \mid 1 \leq i \leq g\}$ be a symplectic basis of $H^1_{\Z}$, i.e. 
$\langle e_i, e_j \rangle = \langle f_i, f_j \rangle=0$, $\langle e_i, f_j \rangle = \delta_{ij}$. 
Under the identification 
$$J_k(\jac(X)) \cong {\rm Hom}(\wedge^{2k+1} H^1_{\Z}, \R/\Z), $$
if $\Phi_1(X_e - X_e^-)$ is non-torsion, there exists elements $\varphi_1$, $\varphi_2$, $\varphi_3 \in S$   
such that 
$$\Phi_1(X_e - X_e^-)( \varphi_1 \wedge \varphi_2 \wedge \varphi_3)$$
 is non-torsion. 
%where $\varphi_i \in S$.  
By renumbering, we may assume that $\varphi_1$, $\varphi_2$, $\varphi_3 \in \{e_i, f_i \mid 1 \leq i \leq 3\}$.  %without loss of generality.     
For $i=1, \ldots,  k-1$, we put 
$$\varphi_{2i+2}=e_{i+3}, \quad \varphi_{2i+3}=f_{i+3}.$$
Note that $i+3 \leq g$ by the assumption.   
Put $\varphi= \varphi_1 \wedge \cdots \wedge \varphi_{2k+1}$. 
 Then, by \cite[ Proposition 3.7]{Otsubo}, we have
 \begin{align*}
 &k! \cdot \Phi_k(W_{k, e}-W_{k, e}^-)(\varphi) \\
 &=k! \cdot \sum_{\sigma} \Phi_1(X_e- X_e^-)(\varphi_{\sigma(1)} \wedge \varphi_{\sigma(2)} \wedge \varphi_{\sigma(3)}) \prod_{i=1}^{k-1}
 \langle \varphi_{\sigma(2i+2)}, \varphi_{\sigma(2i+3)}\rangle \\
& =k! \cdot \Phi_1(X_e- X_e^-)(\varphi_{1} \wedge \varphi_{2} \wedge \varphi_{3}), 
 \end{align*}
 where $\sigma$ runs through the elements of the symmetric group $S_{2k+1}$ such that $\sigma(1) < \sigma(2) < \sigma(3)$, $\sigma(2i+2) < \sigma(2i+3)$ for $1 \leq i \leq k-1$, and 
 $\sigma(2i+2) < \sigma(2i+4)$ for $1 \leq i \leq k-2$. 
Therefore, $\Phi_k(W_{k, e}-W_{k, e}^-)$ is non-torsion. 
 \end{proof}

\begin{cor}
Let $N$ be an integer which has a prime divisor $p \geq 7$ and $X=F_N$ be the Fermat curve of degree $N$.   
% %Then the harmonic volume of $C^{a, b}_N$ has infinite order.  
%%In particular, 
Then $\Phi_k(W_{k, e}-W_{k, e}^-)$ is non-torsion 
for any $e \in F_N$ and $k=1, \ldots, g-2$. 
\end{cor}

\begin{proof}
By Proposition \ref{red}, we are reduced to the case $k=1$, which is a theorem of 
Kimura \cite[Section 4]{Kimura} if $p=7$ and Eskandari and Murty \cite[Theorem 1.1]{EM2} if $p>7$. 
\end{proof}

\subsection{Harris-Pulte formula}
In this subsection, we recall the Harris-Pulte formula, which is a relation between the Abel-Jacobi image of the Ceresa cycle and an extension class of mixed Hodge structures on the space of quadratic iterated integrals on the curve $X$. 
%From now on, let $X$ be a smooth projective curve of genus $g \geq 3$ over $\C$ and $\jac(X)$ be its Jacobian variety. 
%Choosing a base point $e \in X$, $X$ is embedded into $\jac(X)$ sending $e$ to zero. 
%It induces isomorphisms
%$$H_1(X) \xrightarrow{\simeq} H_1(\jac(X)), \quad H^1(\jac(X)) \xrightarrow{\simeq} H^1(X),$$
%which do not depend on the choice of $e$. We identify these and denote them by $H_1$ and $H^1$, respectively.  
%Recall that the cup product induces an isomorphism
%$$\wedge^3H^1 \xrightarrow{\simeq} H^3(\jac(X)). $$
%Let $(H^1 \otimes H^1)'$ be the kernel of the cup product map. 

We put 
$$(H^1 \otimes H^1)' = \operatorname{Ker}(\cup \colon H^1\otimes H^1 \to H^2(\jac(X))). $$
%$$\cup \colon H^1 \otimes H^1 \to H^2(\jac(X)). $$
Then the map 
$$\phi \colon H^1 \otimes (H^1 \otimes H^1)' \to \wedge^3 H^1, $$
which is obtained by restricting the natural quotient map $(H^1)^{\otimes 3} \to \wedge^3H^1$, 
is surjective (\cite[Lemma 4.7]{Pulte}), and  
induces the injective map
$$\phi^* \colon {\rm Ext}_{{\rm MHS}(\Z)}(\wedge^3H^1, \Z(-1)) \to {\rm Ext}_{{\rm MHS}(\Z)}(H^1 \otimes (H^1 \otimes H^1)', \Z(-1)). $$

Let $\pi_1(X, e)$ be the fundamental group. Let $I$ be the augmentation ideal of the group ring $\Z[\pi_1(X, e)]$, i.e.  the kernel of the degree map
$$\Z [\pi_1(X, e)] \to \Z; \quad \sum n_i \gamma_i  \mapsto \sum n_i. $$
By Chen's $\pi_1$-de Rham theorem, $\operatorname{Hom}(\Z[\pi_1(X, e)]/I^{s+1}, \R)$ is generated by closed iterated integrals of length $\leq s$. 
Using this, Hain \cite{Hain} defines a $\Z$-mixed Hodge structure on $\Z[\pi_1(X, e)]/I^s$  such that the natural map $\Z[\pi_1(X, e)]/I^s \to \Z[\pi_1(X, e)]/I^t$ for $s \geq t$ is a morphism of mixed Hodge structures.   
Consider the exact sequence of mixed Hodge structures
\begin{align} \label{exact}
0 \to I^2/I^3 \to I/I^3 \to I/I^2 \to 0.
\end{align}
The map $\pi_1(X, e) \to I/I^2; \gamma \mapsto \gamma -1$ is well-defined and induces an isomorphism
$$H_1(X, \Z) \xrightarrow{\simeq} I/I^2$$
of Hodge structures of weight $-1$. On the other hand, the multiplication $I/I^2 \otimes I/I^2 \to I^2/I^3$ induces an isomorphism
$$\operatorname{Hom}(I^2/I^3, \Z) \xrightarrow{\simeq} (H^1 \otimes H^1)'$$ 
of Hodge structures of weight $2$. 
Taking the dual of \eqref{exact}, we have an exact sequence
 $$0\to  H^1  \to L_2(X, e)\to  (H^1 \otimes H^1)' \to 0,$$
 where we put $L_2(X, e)=\operatorname{Hom}(I/I^3, \Z)$. 
Let $\infty \neq e$ be another point on $X$. 

Put $U=X-\{\infty\}$. We identify $H^1(U)$ and  $H^1$ via the map induced by the inclusion $U \subset X$. 
Then we can obtain an exact sequence of mixed Hodge structures  
$$0\to  H^1  \to L_2(U, e)\to  H^1 \otimes H^1 \to 0$$
similarly as above. 
We have a commutative diagram 
$$
   \xymatrix{
0 \ar[r] & H^1\ar[r]  \ar@{}[d]|*{\parallel} &L_2(X, e) \ar[r]  \ar@{}[d]|*{\cap} & (H^1 \otimes H^1)' \ar[r] \ar@{}[d]|*{\cap} &0\\ 
0 \ar[r] & H^1\ar[r] &L_2(U, e) \ar[r]& H^1 \otimes H^1 \ar[r]&0.
}
$$
%where the middle vertical inclusion is induced by the inclusion $U \subset X$. 
Let $\mathbb{E}_e$ (resp. $\mathbb{E}_e^{\infty}$) be an extension class of the top (resp. bottom) row. 
We regard 
$\mathbb{E}_e$ as an element of 
\begin{align*}
{\rm Ext}_{{\rm MHS}(\Z)}( (H^1 \otimes H^1)', H^1) &\cong {\rm Ext}_{{\rm MHS}(\Z)}((H^1)^{\vee} \otimes (H^1 \otimes H^1)', \Z(0)) \\
&\cong {\rm Ext}_{{\rm MHS}(\Z)}(H^1 \otimes (H^1 \otimes H^1)', \Z(-1)),
\end{align*}
and $\mathbb{E}_e^{\infty}$ as an element of 
\begin{align*}
{\rm Ext}_{{\rm MHS}(\Z)}( H^1 \otimes H^1, H^1) &\cong {\rm Ext}_{{\rm MHS}(\Z)}((H^1)^{\vee} \otimes H^1 \otimes H^1, \Z(0)) \\
&\cong {\rm Ext}_{{\rm MHS}(\Z)}(H^1 \otimes H^1 \otimes H^1, \Z(-1)). 
\end{align*}
 Here we used the Poincar\'e duality $H^1(1) \cong (H^1)^{\vee}$. 
 One sees that $\mathbb{E}_e$ is the restriction of $\mathbb{E}_e^{\infty}$ to $H^1 \otimes (H^1 \otimes H^1)'$. 
 Then Harris's formula \cite[Section 4]{Harris1}, reworked by Pulte \cite[Theorem 4.10]{Pulte}, is
\begin{align*} %\label{AJ}
\phi^* \circ \Phi_1(X_e-X^-_e) =2 \mathbb{E}_e
\end{align*}
under the identification $J_1(\jac(X))={\rm Ext}_{{\rm MHS}(\Z)}(\wedge^3 H^1, \Z(-1))$.

\subsection{The decomposition of $(H^1)^{\otimes 3}$}
In this subsection, for a $\Z$-mixed Hodge structure $H$, we consider the image of $H$ under the forgetful functor ${\rm MHS}(\Z) \to {\rm MHS(\Q)}$, which we denote by the same letter. 
The Hodge structure $(H^1)^{\otimes 3}$ can be decomposed in MHS($\Q$) as follows.  
Let $\xi_{\Delta} \in H^1 \otimes H^1$ be the K\"{u}nneth component of the Hodge class of the diagonal of $X$
 in $H^2(X \times X)$. 
Then we have a decomposition
 $$H^1 \otimes H^1 \otimes H^1 =(H^1 \otimes \langle \xi_{\Delta} \rangle ) \oplus (H^1 \otimes (H^1 \otimes H^1)').$$
Since the Mumford-Tate group of $H^1$ is reductive, the map $\phi$ admits a section $\sigma$ in MHS($\Q$), and we have
 $$H^1 \otimes (H^1 \otimes H^1)' = \operatorname{ker}(\phi) \oplus \sigma(\wedge^3H^1). $$
 Let $\overline{\xi}_{\Delta}$ be the image of $\xi_{\Delta}$ in $\wedge^2H^1$. 
 Then we have a decomposition in MHS($\Q$)
 \begin{align*} %\label{dec}
 H^1 \otimes H^1 \otimes H^1 = (H^1 \otimes \langle \xi_{\Delta} \rangle ) \oplus \operatorname{ker}(\phi) \oplus \sigma(H^1 \wedge \langle \overline{\xi}_{\Delta} \rangle) \oplus \sigma((\wedge^3H^1)_{\rm prim}),
 \end{align*}
 where the last summand (primitive part) is the kernel of the map $\wedge^3H^1 \to \wedge^{2g-1}H^1$ given by wedging by $\overline{\xi}_{\Delta}^{g-2}$ (cf. \cite[Section 4.2]{EM}).
 We put $\mathbb{E} :=\mathbb{E}_e |_{\sigma((\wedge^3H^1)_{\rm prim})}$. 
 Then $\mathbb{E}$ is independent of the choice of $e$ (\cite[Theorem 3.9]{Pulte} and \cite{Harris1}).  
%The restriction of $\mathbb{E}_e^{\infty}$ to each summand is given by the following lemmas.    
%\begin{lem}[{\cite[Theorem 1.2]{Ka}}] \label{lem1}
%The restriction of $\mathbb{E}_e^{\infty}$ to $H^1 \otimes \langle \xi_{\Delta} \rangle$ under the identification (cf. \cite[Section 4.3.1]{EM})
%$${\rm Ext}_{{\rm MHS}(\Q)}(H^1 \otimes  \langle \xi_{\Delta} \rangle, \Q(-1)) \cong \Ch_0(X)_{\rm hom} \otimes \Q$$
%equals $-2g[e]+2[ \infty] +K$, 
%where $K$ is the canonical divisor of $X$. 
%\end{lem}
%
%\begin{lem}[{\cite[Section 3]{Harris1}, \cite[Theorem 4.10]{Pulte} }] \label{lem2}
%The restriction of $\mathbb{E}_e^{\infty}$
% to $\operatorname{ker}(\phi)$ is zero. 
% \end{lem}
% 
% \begin{lem}[{\cite[Corollary 6.7]{Pulte}}] \label{lem3}
%The restriction of $\mathbb{E}_e^{\infty}$ to $\sigma(H^1 \wedge \langle \overline{\xi}_{\Delta} \rangle)$ under the identification (cf. \cite[Section 4.3.3]{EM})
%$${\rm Ext}_{{\rm MHS}(\Q)}(\sigma(H^1 \wedge \langle \overline{\xi}_{\Delta} \rangle), \Q(-1)) \cong \Ch_0(X)_{\rm hom} \otimes \Q$$
%equals  
%$(2g-2)[e]-K$.  
%\end{lem}

\begin{prop} \label{prop} \ 
\begin{enumerate}
\item 
Suppose that $-2g[\infty]+2[e] +K=0$.    
Then $\mathbb{E}_{e}^{\infty} =0$ if and only if $\mathbb{E}_{e} =0$. 
%$$\mathbb{E}_{e}^{\infty} =0 \iff \mathbb{E}_{e}^{\infty} =0$$
\item Suppose that $(2g-2)[e]-K =0$.  
Then $\mathbb{E}_{e}=0$ if and only if $\mathbb{E}=0$.  
%In particular, $\mathbb{E}_e \neq 0$ is independent of the choice of $e$.  
%  \begin{enumerate}
% % \item $\mathbb{E}_{e}^{\infty} \in {\rm Ext}_{{\rm MHS}(\Z)}(H^1 \otimes  H^1 \otimes H^1, \Z(-1))$ is torsion, 
%  \item $\mathbb{E}_{e}^{\infty} \in {\rm Ext}_{{\rm MHS}(\Q)}(H^1 \otimes  H^1 \otimes H^1, \Q(-1))$ is zero,
%   \item $\mathbb{E}_{e} \in {\rm Ext}_{{\rm MHS}(\Q)}(H^1 \otimes  (H^1 \otimes H^1)', \Q(-1))$ is zero,
% \item $\mathbb{E}_e^{\infty} |_{\sigma((\wedge^3H^1)_{\rm prim})}$ is torsion.
%\end{enumerate}
\end{enumerate}
\end{prop}

%\begin{rmk}
%Harris \cite{Harris1} shows that $\mathbb{E}$ is independent of the choice of $e$. 
%Therefore, Proposition \ref{prop} implies that $\mathbb{E}_e^{\infty} \neq 0$ is independent of $e$. 
%\end{rmk}

\begin{proof}
(i) The statement follows from that 
 $\mathbb{E}_e^{\infty}|_{H^1 \otimes (H^1 \otimes H^1)'}=\mathbb{E}_e$ and a result of Kaenders {\cite[Theorem 1.2]{Ka}} that 
$$\mathbb{E}_e^{\infty}|_{H^1 \otimes \langle \xi_{\Delta} \rangle}=-2g[\infty]+2[e] +K$$
 under the identification (cf. \cite[Section 4.3.1]{EM})
$${\rm Ext}_{{\rm MHS}(\Q)}(H^1 \otimes  \langle \xi_{\Delta} \rangle, \Q(-1)) \cong \Ch_0(X)_{\rm hom} \otimes \Q. $$
%equals $-2g[e]+2[ \infty] +K$. 

(ii) The statement follows from that results of Harris {\cite[Section 3]{Harris1} and Pulte \cite[Theorem 4.10]{Pulte} }] that $\mathbb{E}_e|_{\operatorname{ker}(\phi)}=0$, and Pulte {\cite[Corollary 6.7]{Pulte}} that 
$$\mathbb{E}_e|_{\sigma(H^1 \wedge \langle \overline{\xi}_{\Delta} \rangle)}=(2g-2)[e]-K$$
under the identification (cf. \cite[Section 4.3.3]{EM})
$${\rm Ext}_{{\rm MHS}(\Q)}(\sigma(H^1 \wedge \langle \overline{\xi}_{\Delta} \rangle), \Q(-1)) \cong \Ch_0(X)_{\rm hom} \otimes \Q. $$
%equals  
%$(2g-2)[e]-K$.  
\end{proof}

\subsection{Darmon-Rotger-Sols formula}
 Let $\Delta \in  \Ch_1(X \times X)$ be the diagonal of $X$, and
$$p_i \colon X \times X \to X \quad (i=1, 2)$$
be the projection to the $i$-th component. 
For $Z \in \Ch_1(X \times X)$, put 
\begin{align*}
&Z_{12}=(p_1)_*(Z \cdot \Delta)=(p_2)_*(Z \cdot \Delta), \\ 
&Z_1=(p_1)_*(Z \cdot (X \times \{e\})), \quad Z_2=(p_2)_*(Z \cdot (\{e\} \times X)) \in \Ch_0(X). 
\end{align*}
Put 
$$P_Z=Z_{12}-Z_1-Z_2-(\deg(Z_{12})-\deg(Z_1)-\deg(Z_2))[e] \in \jac(X). $$
Then the point $P_Z$ is related to the extension $\mathbb{E}_e^{\infty}$ as follows. 
Let $\xi_Z$ be the $H^1 \otimes H^1$ K\"{u}nneth component of the class of $Z$ in $H^2(X \times X)$.  
Let 
\begin{align*}
\xi_Z^{-1} \colon {\rm Ext}_{{\rm MHS}(\Z)}((H^1)^{\otimes 3}, \Z(-1)) \to {\rm Ext}_{{\rm MHS}(\Z)}(H^1, \Z(-1)) \cong J_0(X) =\jac(X), 
\end{align*}
where the first arrow is the pull back along the morphism $H^1 \to (H^1)^{\otimes 3}$ defined by $\omega \mapsto \omega \otimes \xi_Z$.  
Then we have the following.  
\begin{prop}[{\cite[Corollary 2.6]{DRS}}]  \label{extmap}
For any $Z \in \Ch_1(X \times X)$, we have 
\begin{align*}  
\xi^{-1}_Z(\mathbb{E}_e^{\infty}) =\left(\int_{\Delta}\xi_Z \right) ([\infty] -[e])-P_Z
\end{align*}
in $\jac(X)$. 
\end{prop}

\section{Proof of Theorem \ref{main}}
There are $3N$ points on $F_N$ 
$$P_i=(0 : \zeta_N^i : 1), \quad Q_i= (\zeta_N^i : 0 : 1), \quad R_i =(\xi_N \zeta_N^i : 1: 0), \quad (i \in \Z/N\Z) $$ 
where we put $\xi_N=\exp(\pi i /N)$.    
Fix $P_0$ as the base point, then the above points are torsion points in $\jac(F_N)$ \cite{GR}. 
Therefore for the base point $\pi_N^{a, b}(P_0)$, the images of these points under $(\pi_N^{a, b})_*$ are also torsion in $\jac(C_N^{a, b})$. 
We shall continue to use the notation as in the previous section, specializing $X=C_N^{a, b}$, $e=\pi_N^{a, b}(P_0)$ and $\infty = \pi_N^{a, b}(Q_0)$.

\begin{lem} \label{can}
Let $K_C$ (resp. $g$) be the canonical divisor (resp. genus) of $C_N^{a, b}$. 
Then $K_C-(2g-2)[e]$, $K_C-2g[\infty]+2[e] \in \jac(C_N^{a, b})$ are torsion points. 
\end{lem}
\begin{proof}
Since 
$$K_C-2g[\infty]+2[e]=K_C -(2g-2)[e]-2g([\infty]-[e])$$
and $[\infty]-[e]$ is a torsion point, it suffices to show that $K_C -(2g-2)[e]$ is a torsion point. 
Let $K_F$ be the canonical divisor of $F_N$ and $R_{\pi_N^{a, b}}$ be the ramification divisor of $\pi_N^{a, b}$, i.e.  
\begin{align*}
&K_F=(N-1) \sum_{i=0}^{N-1}Q_i -2 \sum_{i=0}^{N-1} R_i, \\
&R_{\pi_N^{a, b}}=(\gcd(N, a)-1)\sum_{i=0}^{N-1} P_i + (\gcd(N, b)-1)\sum_{i=0}^{N-1} Q_i +(\gcd(N, a+b)-1)\sum_{i=0}^{N-1} R_i. 
\end{align*}
%Here we denote $P_i=(0 : \zeta_N^i : 1)$, $Q_i= (\zeta_N^i : 0 : 1)$ and $R_i =(\xi_N \zeta_N^i : 1: 0)$, where we put $\xi_N=\exp(\pi i /N)$.    
Then we have 
$$K_F= (\pi_N^{a, b})^*(K_C) + R_{\pi_N^{a, b}}$$
up to principal divisor (cf. \cite[Proposition 2.3, Chap.IV]{Hartshorne}). 
Therefore we have
\begin{align*}
N(K_C- (2g-2)[e])= (\pi_N^{a, b})_*\left( K_F -R_{\pi_N^{a, b}}-(2g-2)N[P_0]\right)  \label{eq11} 
\end{align*}
in $\jac(C_N^{a, b})$. 
Since $P_i$, $Q_i$ and $R_i$ are torsion in $\jac(F_N)$,  
$K_F -R_{\pi_N^{a, b}}-(2g-2)N[P_0]$ is torsion, 
which finishes the proof.   
\end{proof}

%\begin{lem}
%The proof of Theorem \ref{main} is reduced to the case when $N=p$. 
%\end{lem}
%\begin{proof}
\begin{proof}[Proof of Theorem \ref{main}]
First, 
by Proposition \ref{red}, it suffices to show the case when $k=1$. 
Secondly, 
consider the map
 $$f \colon F_N \to  F_p; \quad (x_0 : y_0 : z_0) \mapsto (x_0^{N/p} : y_0^{N/p} : z_0^{N/p}). $$
 Let $\langle a \rangle \in \{0, \ldots, p-1\}$ be the representative of $a$.  
 Then $f$ descends to a map $\overline{f} \colon C^{a, b}_N \to C^{\langle a \rangle,  \langle b \rangle}_p$. 
 Since 
 $$f_{*}(\Phi_1(C_{N, e}^{a, b}-(C^{a, b}_{N,e})^{-}))=\deg \overline{f} \cdot \Phi_1 \left(C_{p, \overline{f}(e)}^{a, b}-(C^{a, b}_{p,\overline{f}(e)})^{-} \right), $$
we are reduced to the case when $N=p$. 
When $p=7$, Theorem \ref{main}  
is proved in \cite[Section 4]{Kimura}, hence  
we may assume that $p > 7$.  

By Lemma \ref{can} and Proposition \ref{prop}, it suffices to show that, for the specific choices of $e$ and $\infty$ as above,  
the element $\mathbb{E}^{\infty}_{e} \in {\rm Ext}_{{\rm MHS}(\Q)}(H^1 \otimes  H^1 \otimes H^1, \Q(-1))$ is nonzero.  
%First, we reduce the proof to the case when $N=p$. 
%By Proposition \ref{prop}, 
%it suffices to show that if the image of the Ceresa cycle of $C_p^{a, b}$ under the Abel-Jacobi map is of infinite order,  
By  Lemma \ref{alpha}, the automorphism $\beta$ of  $F_p$ descends to an automorphism $\widetilde{\beta}$ of $C_p^{a, b}$; let 
$Z$ be the graph of $\widetilde{\beta}$.  
Since $[\infty] - [e]$ is torsion, it suffices to show that $P_Z$ is non-torsion by Proposition \ref{extmap}.    
%By Proposition \ref{prop}, we show that the extension $\mathbb{E}_{\infty}^{e} \in {\rm Ext}_{{\rm MHS}(\Z)}(H^1 \otimes  H^1 \otimes H^1, \Z(-1))$ is of infinite order. 
%Since $[\infty] - [e]$ is a torsion point, it suffices to show that there exists $Z \in \Ch_1(C_p^{a, b} \times C_p^{a, b})$ such that $P_Z$ is of infinite order by Lemma \ref{extmap}.    
 %By  Lemma \ref{alpha}, the automorphism $\alpha$ of  $F_p$ descends to the automorphism $\widetilde{\alpha}$ of $C_p^{a, b}$ and let 
%$Z$ be the graph of $\widetilde{\alpha}$.  
 %We show that $P_Z$ is of infinite order. 
Since $\widetilde{\beta} \circ \pi_p^{a, b}=\pi_p^{a, b}\circ \beta$ and $\widetilde{\beta}$ has two fixed points by Lemma \ref{fix}, 
we have
\begin{align*}
 P_Z&=([\pi_p^{a, b}(P)]+[\pi_p^{a, b}(\overline{P})]-2[e])-([\widetilde{\beta}(e)]+[\widetilde{\beta}^{-1}(e)]-2[e]) \\
&=(\pi_p^{a, b})_*(([P]+[\overline{P}]-2[O])-([\beta(O)]+[\beta^{-1}(O)]-2[O])).  %\\
\end{align*}
The point $[\beta(O)]+[\beta^{-1}(O)]-2[O]$ is a torsion point on $\jac(F_p)$, hence 
$(\pi_p^{a, b})_*([\beta(O)]+[\beta^{-1}(O)]-2[O])$ is a torsion point on $\jac(C_p^{a, b})$. 
On the other hand, since
$a-b, a+2b, 2a+b \not \equiv 0 \pmod{p}$ by the assumption $a^2+ab+b^2 \equiv 0 \pmod{p}$, 
the point 
$$(\pi_p^{a, b})_*([P]+[\overline{P}]-2[O]) \in \jac(C_p^{a, b}) $$
is non-torsion by Theorem \ref{GR}. 
Therefore    
the point $P_Z$ is non-torsion, which finishes the proof. 
\end{proof}

\section*{Acknowledgment}
The author would like to sincerely thank Noriyuki Otsubo for valuable discussions and his careful reading on a draft of this paper.  
He would like to thank Yoshinosuke Hirakawa for valuable discussions and many helpful comments. 
He also thanks Yuki Goto and Ryutaro Sekigawa for valuable discussions. 
This paper is a part of the outcome of research performed under Waseda University Grant for Special Research Projects (Project number: 2023C-274) and Kakenhi Applicants (Project number: 2023R-044).

\end{document}